\documentclass[11pt]{article}
\usepackage{}
\usepackage[total={6in, 9in}]{geometry}
 \usepackage{amsfonts}
\usepackage{amsthm}
\usepackage{amssymb}
\usepackage{amsmath,amsfonts}
\usepackage{mathrsfs}
\usepackage{cases}
\usepackage{latexsym,bm}
\usepackage{indentfirst}
\usepackage{color}
\usepackage{ifpdf}
\usepackage{graphicx}
\usepackage{psfrag}
\usepackage{enumerate}
\usepackage{dsfont}
\usepackage[
pdfauthor={},
pdftitle={},
pdfstartview=XYZ,
bookmarks=true,
colorlinks=true,
linkcolor=blue,
urlcolor=blue,
citecolor=blue,
bookmarks=true,
linktocpage=true,
hyperindex=true
]{hyperref}

\usepackage[natural]{xcolor}
\usepackage{tikz}
\usepackage{subfigure,tikz}
\usepackage{tikz-3dplot}
\usepackage{pgf,tikz,pgfplots}
\usepackage{mathrsfs}

\newtheorem{THM}{\textbf{Theorem}}

\newtheorem{LEM}{\textbf{Lemma}}
\newtheorem{CLA}{\textbf{Claim}}[section]
\newtheorem{CON}[THM]{\textbf{Conjecture}}

\newcommand{\arxiv}[1]{\href{http://arxiv.org/abs/#1}{\texttt{arXiv:#1}}}

\begin{document}
\title{Hamilton cycles in  tough $(2P_2 \cup P_1)$-free graphs}
\author{Songling Shan\footnote{Auburn University, Department of Mathematics and Statistics, Auburn, AL 36849.
		Email: {\tt szs0398@auburn.edu}.   
		Supported in part by NSF grant DMS-2345869.}
		\qquad 
		 Arthur Tanyel\footnote{Auburn University, Department of Mathematics and Statistics, Auburn, AL 36849.
		 	Email:	{\tt ant0034@auburn.edu}. }
}

\date{\today}
\maketitle

\begin{abstract}
In 1973, Chv\'atal conjectured that there exists a constant $t_0$ such that every $t_0$-tough graph on at least three vertices is Hamiltonian. While this conjecture is still open, work has been done to confirm it for several graph classes, including all $F$-free graphs for every 5-vertex linear forest $F$ other than $P_5$ and $2P_2\cup P_1$. In this note, we show that  11-tough $(2P_2 \cup P_1)$-free graphs on at least three vertices are Hamiltonian.

\medskip 


\end{abstract}

\section{Introduction}

All graphs in this paper are simple. 
Let $G$ be a graph.
We use $V(G)$ and $E(G)$ to denote the vertex set and edge set of $G$, respectively.
For a vertex $v \in V(G)$, we denote by $N_G(v)$ the set of neighbors of $v$.
For $S \subseteq V(G)$, let $N_G(S)=(\bigcup_{x\in S} N_G(x))\,\setminus S$. 
The degree of $v$ in $G$ is denoted by 
  $d_G(v)$, and let $d_G(v,S) = |N_G(v) \cap S|$ for $S\subseteq V(G)$.  
For $H\subseteq G$, we write $d_G(v,H)$ for $d_G(v, V(H))$. 
For  $S \subseteq V(G)$, 
we let $G[S]$ denote the subgraph of $G$ induced on $S$, and let 
 $G - S = G[V(G) \setminus S]$.
We write $G - v$  for  $G - \{v\}$.
We write $u \sim v$ if two vertices $u$ and $v$ are adjacent in $G$ and $u \nsim v$ otherwise.
Let $V_1,V_2 \subseteq V(G)$ be disjoint. Then $E_G(V_1,V_2)$ denotes the set of edges of $G$
with one endvertex in $V_1$ and the other in $V_2$, and 
 $G[V_1,V_2]$   denotes the bipartite subgraph of $G$ with bipartitions $V_1$ and $V_2$ and  edge set $E_G(V_1,V_2)$.

Given a graph $R$, we say that  $G$ is \emph{$R$-free} if $G$ does not contain $R$ as an induced subgraph. 
For an integer $k\ge 2$, we use $kR$ to denote the disjoint union of $k$ copies of $R$. When we say that $G$ is $(R_1\cup R_2)$-free,  we  take  $(R_1\cup R_2)$ as the vertex-disjoint union of 
two graphs $R_1$ and $R_2$. We use $P_n$ to denote a path on $n$ vertices.   

A graph is called \emph{Hamiltonian} if it contains a Hamilton cycle, and \emph{Hamiltonian-connected} if there exists a Hamilton path between every pair of distinct vertices. In 1973, Chv\'atal studied Hamilton cycles through the notion of \emph{toughness}~\cite{chvatal-tough-c}.
Let $c(G)$ denote the number of components of a graph $G$, and let $t \ge 0$ be a real number. The \emph{toughness} of a graph $G$, denoted by $\tau(G)$, is defined as
$$
\tau(G) = \min\left\{\frac{|S|}{c(G - S)} : S \subseteq V(G),\ c(G - S) \ge 2\right\},
$$
if $G$ is not complete; otherwise, $\tau(G) = \infty$.  
A graph $G$ is said to be \emph{$t$-tough} if $\tau(G) \ge t$.
In~\cite{chvatal-tough-c}, Chv\'atal proposed the following conjecture.

\begin{CON}[Chv\'atal's toughness conjecture]
    There exists a constant $t_0$ such that every $t_0$-tough graph on at least three vertices is Hamiltonian.
\end{CON}

 Bauer, Broersma and Veldman~\cite{Tough-counterE} have constructed
 $t$-tough graphs that are not Hamiltonian for all $t < \frac{9}{4}$, so
 $t_0$ must be at least $\frac{9}{4}$ if Chv\'atal's toughness conjecture is true.
 The conjecture has
 been verified  for certain classes of graphs including 
 planar graphs, claw-free graphs, co-comparability graphs, and
 chordal graphs. For a more comprehensive list of graph classes for which the conjecture holds, see  the survey article by Bauer, Broersma, and Schmeichel~\cite{Bauer2006}  in 2006.  Some recently established families of 
 graphs for which the conjecture holds include $2K_2$-free graphs~\cite{2k2-tough,1706.09029,OS2021}, 
 and $R$-free 
 graphs if $R$ is a 4-vertex linear forest~\cite{MR3557210} or  $R\in \{P_4 \cup P_1, P_2\cup P_3, P_3\cup 2P_1,  P_2\cup 3P_1, P_2\cup kP_1\}$~\cite{shan2025p4unionp1, S2021, P3-union-2P1,  HG2021, MR4455928,  MR4676626, MR4658427}, where  
 $k\ge 4$ is an integer. In general, the conjecture is still wide open. 
We obtain the following result, which,  when combined with previous results, confirms Chv\'atal's toughness conjecture for all $R$-free graphs where $R$ is a  five-vertex linear forest except $P_5$.

\begin{THM}\label{main}
    	Every 11-tough $(2P_2 \cup P_1)$-free graph  on at least three vertices is Hamiltonian.
\end{THM}

We provide some preliminaries in the next section. The proof of Theorem~\ref{main} is given in the last section. 
\section{Preliminaries}

We begin this section with some definitions.
A \textit{star-matching} in a graph is the union of some vertex-disjoint copies of stars.
The  vertices of degree greater than one in a star-matching are called the \textit{centers} of the star matching.
For an integer $t \ge 1$, we call a star-matching $M$ a $K_{1,t}$-matching if every star in $M$ is isomorphic to $K_{1,t}$.
For a star-matching $M$, if $x,y \in V(M)$  with  $xy \in E(M)$ and $d_M(x)=1$,  then we say  that $x$  is  a   \textit{partner} of $y$ under $M$.

A path $P$ connecting two vertices $u$ and $v$ is called 
a {$(u,v)$-path}, and we write $uPv$ or $vPu$ in order to specify the two endvertices of 
$P$. 
Let $uPv$ and $xQy$ be two paths. If $vx$ is an edge, 
we write $uPvxQy$ as
the concatenation of $P$ and $Q$ through the edge $vx$.

The  first a few  lemmas  below provide sufficient conditions for the existence of a Hamilton cycle in  a graph. 

\begin{LEM}[\cite{D1952}]\label{dirac}
    Let $G$ be a graph on $n \ge 3$ vertices.
    If $\delta(G) \ge \frac{n}{2}$, then $G$ is Hamiltonian.
\end{LEM}

\begin{LEM}[\cite{MR1336668}]\label{bauer_min_degree}
	Let $t>0$ be a real number and $G$ a $t$-tough graph on $n \ge 3$ vertices with $\delta(G) > \frac{n}{t+1}-1$.
	Then $G$ is Hamiltonian.
\end{LEM}

The result below says that a cycle  $C$ of a graph $G$ can be extended to include vertices that have many neighbors from $V(C)$
in terms of a function involving the  toughness of $G$. 

\begin{LEM}[{\cite[Lemma 2.16]{shan2025p4unionp1}}]\label{lem:vertex-insetting}
	Let $t>0$ and $G$ be a $t$-tough  $n$-vertex  graph with a non-Hamiltonian cycle  $C$. For a connected subgraph  $H$ of $G-V(C)$, if 
	$|N_G(V(H))\cap V(C)|> \frac{n}{t+1}-1$, then we can extend $C$ to a cycle $C^*$ 
	such that  $V(C) \subseteq V(C^*)$ and $V(C^*)\cap V(H) \ne \emptyset$. 
\end{LEM}

Let $G$ be a graph. 
Two edges of $G$ are \emph{independent} if they do not share any endvertices.    The connectivity 
of  $G$, denoted $\kappa(G)$, is the size of a minimum cutset of $G$ if $G$ is noncomplete, and 
 is defined as $ |V(G)|-1$  if $G$ is complete. 
The size of a maximum independent set in $G$  is called the \emph{independence number} of $G$, and is denoted by $\alpha(G)$. 
We need the following 1982-result by H\"{a}ggkvist and  Thomassen. 

\begin{LEM}[\cite{Cycles-through-edges}]\label{haggkvist_cycles}
	Let $G$ be a graph and $L\subseteq E(G)$ be a set of independent edges. 
	If  $ \kappa(G) \ge |L|+\alpha(G)$,  then $G$ has  a Hamilton cycle  containing every edge of $L$.
\end{LEM}

Let  $G$ be a graph and $S$ be a cutset of $G$.
A component of $G-S$ is  \textit{trivial} if it contains only one vertex; otherwise,  it is \textit{nontrivial}.
If  $G$ is  $(2P_2\cup P_1)$-free   and $G-S$ has a nontrivial component,  then every 
other component of $G-S$ is $(P_2\cup P_1)$-free.  Note  that if $G$ is $(P_2\cup P_1)$-free, then it is also $P_4$-free. 
We also need the following result by Jung from 1978 on $P_4$-free graphs, and some properties of $(P_2\cup P_1)$-free graphs. 
To state the result by Jung, we  define the scattering number of a graph.

Let $G$ be a graph. We call 
$$
s(G) =\max\{c(G-S)-|S|: S\subseteq V(G), c(G-S) \ge 2\} 
$$
the \emph{scattering number} of $G$ if $G$ is not complete; otherwise $s(G)=\infty$. A set $S\subseteq V(G)$ with $c(G-S)-|S|=s(G)$ and $c(G-S) \ge 2$ 
is called a \emph{scattering set} of $G$.

\begin{LEM}[\cite{P4-free-1-tough}]\label{jung}
	Let $G$ be a $P_4$-free graph. Then $G$ is Hamiltonian-connected if and only if $s(G) < 0$.
\end{LEM}

\begin{LEM}\label{propertiesofP2cupP1freegraphs}
	Let $G$ be a   $(P_2 \cup P_1)$-free graph.  Then the following statements hold. 
	
	\begin{enumerate}
		\item  If $S$ is a cutset of $G$, then every component of $G-S$ is trivial.
		
		\item If  $S$ is a minimal cutset of $G$, then  $G[S,V(G)\setminus S]$ is a complete bipartite graph.
		
		\item $\kappa(G)=\delta(G)$. 
		\item $\delta(G) \ge |V(G)|-\alpha(G)$. 
	\end{enumerate}
\end{LEM}

\proof Statement (1) is trivial by the $(P_2 \cup P_1)$-freeness of $G$. 

For (2),  suppose  otherwise that there exist $x \in S$ and $y \in V(G)\setminus S$ such that $x \nsim y$.   
Then by (1),  $y$  is the vertex of a trivial component of $G-S$. Thus $S\setminus\{x\}$ 
is also a cutset of $G$, a contradiction to $S$ being a minimal cutset of $G$.

For (3), we already have $\kappa(G) = \delta(G)$ if $G$ is a complete graph.   Thus we assume that $G$ is noncomplete. Let $W$ be a minimum cutset of  $G$. Then by (1) and (2), we know that    there is a vertex $v\in V(G)$ with $d(v)=|W|$. Thus  $\delta(G) \le  d(v) =|W|$. As $|W| =\kappa(G)  \le \delta(G)$, it follows that $\kappa(G) =\delta(G)$. 

For (4),  the assertion is clear if $G$ is a complete graph.   Thus we assume that $G$ is noncomplete. Let $W$ be a minimum cutset of  $G$. 
Then we have  that $|W| =\delta(G)$ by (3), and  that each component of $G-W$ is trivial by (1). Thus $\alpha(G) \ge |V(G)|-|W| =|V(G)|-\delta(G)$. 
As a consequence,  $\delta(G) \ge |V(G)|-\alpha(G)$. 
\qed

Lastly, we need the following results on the existence of star-matchings in graphs.  The first one is a generalization of Hall's matching theorem.

\begin{LEM}[{\cite[Theorem 2.10]{AK2011}} ]\label{akiyama_and_kano}
    Let $G$ be a bipartite graph with partite sets $X$ and $Y$, and let $f$ be a function from $X$ to the set of positive integers.
  For every $S \subseteq X$,  if $|N_G(S)| \ge \sum_{v \in S} f(v)$, then $G$ has a subgraph $H$ such that $X \subseteq V(H)$, $d_H(v) = f(v)$ for every $v \in X$, and $d_H(u) = 1$ for every $u \in Y \cap V(H)$.
\end{LEM}

\begin{LEM}\label{lem:star-matching}
	Let $t\ge 1$  be a real number  and $G$ be a $t$-tough noncomplete graph. Then for any independent set $X$ of $G$, there is a $k_{1, \lfloor t \rfloor}$-matching with centers  precisely 
	as vertices of $X$. 
\end{LEM}

\proof  Let $H=G[X,V(G)\setminus X]$ and  $S\subseteq X$ be nonempty.  If $|S|=1$, then we have $|N_H(S)| \ge \delta(G) \ge 2t$. 
Thus we assume that $|S| \ge 2$, and so $N_H(S)$ is a cutset  of  $G$. 
Since $G$ is $t$-tough, we have $|N_H(S)| =|N_G(S)| \ge t|S|$. By Lemma~\ref{akiyama_and_kano}, 
$G$ has a $k_{1, \lfloor t \rfloor}$-matching with centers  precisely 
as vertices of $X$. 
\qed

\section{Proof of Theorem~\ref{main}}

\begin{proof}[Proof of Theorem~\ref{main}]

Let $t \ge 11$ and $G$ be a $t$-tough $(2P_2 \cup P_1)$-free graph on $n\ge 3$ vertices. By Lemma~\ref{bauer_min_degree}, we may assume that $\delta(G) \le \frac{n}{t+1} -1$.  Thus $G$ is not a complete graph  and  $\delta(G) \ge 2t\ge 22$. 
Additionally, by the toughness of $G$, we have  $\alpha(G) \le \frac{n}{t+1}$. 
We consider two cases in  completing the proof.   Roughly,  in both cases,  we split the graph into two parts $R_1$ and $R_2$, where $R_1$
consists of vertices of small degrees in $G$. 
We   ``warp'' vertices of $R_1$  using a union  $\mathcal{Q}$ of  vertex-disjoint  paths with endvertices from $R_2$. Then we construct $R_2^*$ by 
adding a set  $L$ of independent edges  whose endvertices are precisely the endvertices of  the components of $\mathcal{Q}$.  
The graph $R_2^*$ will be shown to have high connectivity. 
Applying Lemma~\ref{haggkvist_cycles} on $R_2^*$, we find a Hamilton cycle $C^*$  of $R^*_2$ 
that contains all edges of $L$.  Now  replacing each edge  $xy$ of $L$ on $C^*$ by the path of $\mathcal{Q}$ with endvertices $x,y$ gives 
a Hamilton cycle of $G$. 

{\bf \noindent Case 1: There exists $uv \in E(G)$ for which $|N_G(u) \cup N_G(v)| \le \frac{5n}{12}$.}

\smallskip 

Let $S=(N_G(u) \cup N_G(v))\setminus \{u,v\}$.  Then  $|S| \le \frac{5n}{12}-2$ and $S$ is a cutset of $G$. Furthermore, we have the following claim. 

\begin{CLA}\label{sisacutset}
We have 	$c(G-S) = 2$. 
\end{CLA}

\begin{proof}[Proof of Claim~\ref{sisacutset}]
	Clearly, $D_1 = G[\{u,v\}]$ is one component of $G-S$.
	As $|S \cup \{u,v\}| \le \frac{5n}{12}$, $c(G - S)\ge 2$. If $G-S$ has another nontrivial component different from $D_1$, 
	  then   by the $(2P_2 \cup P_1)$-freeness of $G$, we have $c(G-S) = 2$. 
	Thus we assume that all  the  components of $G-S$  other than $D_1$ are trivial. 
 This gives $c(G-S) = n - |S| - 1 \ge  \frac{7n}{12}+1$.  However, 
	\[
	\frac{|S|}{c(G-S)} \le \frac{\frac{5n}{12}-2}{\frac{7n}{12}+1}< 11,
	\]
	a contradiction.
\end{proof}

Let $D_1 = G[\{u,v\}]$ and $D_2$ be the other component of $G-S$. 
Since  $G$ is $(2P_2 \cup P_1)$-free and $D_1$ is nontrivial,  it follows that $D_2$ is  $(P_2 \cup P_1)$-free.
Let $$S_1 = \{x \in S : d_G(x,D_2) < \frac{2n}{t+1}\}, \, S_2 = S \setminus S_1, \, G_1 = G[S_1 \cup \{u,v\}],  \, G_2 = G[S_2 \cup V(D_2)].$$

\begin{CLA}\label{G1isP2cupP1free}
The graph 	$G_1$ is $(P_2 \cup P_1)$-free.
\end{CLA}

\begin{proof}[Proof of Claim~\ref{G1isP2cupP1free}]
	Assume  the statement is false.  We let $x,y,z \in S_1 \cup \{u,v\}$ such that $G_1[\{x,y,z\}] = P_2 \cup P_1$.
	Then by the definition of $S_1$, we have $ |N_G(\{x,y,z\}) \cap V(D_2)|< \frac{6n}{t+1}$.
	Since $|V(D_2)| \ge \frac{7n}{12}$, it follows that  $|V(D_2)| -  |N_G(\{x,y,z\}) \cap V(D_2)| >\frac{n}{t+1}$. 
As $ \alpha(D_2) \le \alpha(G) \le \frac{n}{t+1}$, 
	  $V(D_2) \setminus N_G(\{x,y,z\}) $ is not independent in $G$. Thus  there exist $u,v \in V(D_2) \setminus N_G(\{x,y,z\}) $ such that $u \sim v$. However,  $G[\{x,y,z,u,v\}] = 2P_2 \cup P_1$, a contradiction.
\end{proof}

A \emph{path-cover}  $\mathcal{Q}$ of $G_1$  is a  union of  some vertex-disjoint paths  
	such that $V(G_1)\subseteq V(\mathcal{Q})$.   
The path-cover 
	is  called \emph{$W$-matched}  for some $W\subseteq V(G_2)$ if the two endvertices of each path of $\mathcal{Q}$  belong  to $W$.

	\begin{CLA}\label{claim:path-cover}
	The graph $G_1$ has a $W$-matched path-cover $\mathcal{Q}$ for some $W\subseteq V(G_2)$ such that the internal vertices of 
	each path of $\mathcal{Q}$ are all from $V(G_1)$, and $\mathcal{Q}$ has exactly $\max\{1,s(G_1)\}$ components. 
	\end{CLA}
	
	\proof[Proof of Claim~\ref{claim:path-cover}]
		 We construct $\mathcal{Q}$ according to the scattering number of $G_1$. 
	
		Consider first that  $s(G_1) \le -1$. 
	By Claim~\ref{G1isP2cupP1free}, $G_1$ is $(P_2 \cup P_1)$-free.  As $P_2 \cup  P_1 \subseteq P_4$, $G_1$ is $P_4$-free. 
	Thus $G_1$ is Hamiltonian-connected by Lemma~\ref{jung}. 
	As $G$ is 11-tough and $|V(G_1)| \ge 2$, there are distinct $x,y\in V(G_1)$ and distinct $z,w\in V(G_2)$ such that $x\sim z$ and $y\sim w$. 
	Let $P$ be a Hamiltonian $(x,y)$-path in $G_1$. Then $zxPyw$
	is a desired $W$-matched path-cover of $G_1$ with $W=\{z,w\}$.

Suppose then that  $s(G_1)  \ge 0$.  Let $T\subseteq V(G_1)$ be a minimum  cutset of $G_1$. Then by Lemma~\ref{propertiesofP2cupP1freegraphs},  $c(G_1-T)=|V(G_1)\setminus T|$ and $G_1[T,V(G_1)\setminus T]$ is a complete bipartite graph. 
Furthermore, as  $c(G_1-T) -|T| \le s(G_1)$ and $c(G_1-T^*)-s(G_1)=|T^*|   \ge |T| $ for any scattering set $T^*$ of $G_1$,  
  we get $c(G_1-T^*) \ge c(G_1-T)$. As  $c(G_1-T^*)=|V(G_1)\setminus T^*|$ and  $c(G_1-T)=|V(G_1)\setminus T|$, we get  $|T^*| \le |T|$
and so $|T^*|=|T|$.  Thus $T$ is a scattering set of $G_1$, and so  $|T| \le c(G_1-T)$ by  the assumption that $s(G_1)  \ge 0$. 

As $G$ is 11-tough, Lemma~\ref{lem:star-matching} implies that $G$ has a $K_{1,2}$-matching  $M$ with vertices of $V(G_1)\setminus T$
 precisely as its centers. Since $|T| \le |V(G_1)\setminus T|$, at most $\lfloor \frac{1}{2}|V(G_1)\setminus T| \rfloor$ vertices of $V(G_1)\setminus T$
have both their partners as $T$-vertices.   Let $U\subseteq V(G_1)\setminus T$ such that every vertex of $U$ has 
a partner from $T$. Then we have $|U| \le |T|$, and if $|U| =|T|$, then at least one vertex of $U$ has a partner from $V(G_2)$. 
If $|T|<|V(G_1)\setminus T| $, then  let $U^*$ be a subset of $V(G_1)\setminus T$ with $U\subseteq U^*$ such that $|U^*|=|T|+1$. 
Let $x,y\in U^*$ such that both $x$ and $y$ have a partner from $V(G_2)$, where $x$ and $y$ exist by our choice of $U$. 
 We let $P$ be a Hamiltonian $(x,y)$-path of 
$G_1[T,U^*]$ ($G_1[T,U^*]$ is a compelte bipartite graph), and let $z,w$ be respectively a partner of $x$ and $y$ from $V(G_2)$. 
Then with $zxPyw$ and the rest paths of $M$ not containing any vertex of $P$, form a desired 
$W$-matched path-cover of $G_1$ with $W$ as the set of endvertices of these paths.

Thus we assume that $|T|=|V(G_1)\setminus T| $.    If $|V(G_1)|  \le 21$, then as $\delta(G) \ge 22$,   
every vertex of $G_1$ has in $G$ at least two neighbors from $V(G_2)$.     Let $x\in T$ and $y\in V(G_1)\setminus T$, and let $z, w\in V(G_2)$ be distinct such that  $x\sim z$ and $y\sim w$.  Let $P$ be a Hamiltonian $(x,y)$-path of  $G_1$. Then $zxPyw$
is a desired $W$-matched path-cover of $G_1$ with $W=\{z,w\}$.  
Thus we assume that $|V(G_1)|  \ge 22$.  
This implies that $|T|=|V(G_1)\setminus T| \ge 11$. Thus 
there are at least two distinct vertices  $x, y$ of $V(G_1)\setminus T$ 
that each has a partner from $V(G_2)$.  Let $z,w$ be respectively a partner of $x$ and $y$ from $V(G_2)$.

If there are $x_1,x_2\in T$ such that $x_1\sim x_2$, then  $G_1$ has a Hamiltonian-$(x,y)$ path $P$ (the path must contain the edge $x_1x_2$). 
Then $zxPyw$ is a desired 
$W$-matched path-cover of $G_1$ with $W=\{z,w\}$. 

Thus we assume that $T$ is also an independent set in $G$. As $G$ is 11-tough, and $|T|=|V(G_1)\setminus T|$, there exists $x^*\in T$
such that $x^*$ has in $G$ a neighbor  $z^*$ from $V(G_2)$ with  $z^*\ne z$. 
We let $P$ be a Hamiltonian $(x,x^*)$-path of 
$G_1$.  Then $zxPx^*z^*$ is a desired 
$W$-matched path-cover of $G_1$ with $W=\{z,z^*\}$. 
\qed 
 
Let $\mathcal{Q}$ be a  $W$-matched path-cover of $G_1$ for some $W\subseteq V(G_2)$ that is described in Claim~\ref{claim:path-cover}.
Let $L =\{xy: \text{$x$ and $y$ are two endvertices of a component of $\mathcal{Q}$}\}$ be the set of edges of $G\cup \overline{G}$ 
with endvertices as endvertices  of components of $\mathcal{Q}$.  Let $G_2^*$ be obtained from $G_2$ by adding all edges of $L$ 
that are not in $G_2$ already.

 \begin{CLA}\label{KappaofG2star}
 We have 	$\kappa (G_2^*) \ge |L| + \alpha(G_2^*)$. 
 \end{CLA}

 \begin{proof}[Proof of Claim~\ref{KappaofG2star}]
 	By Claim~\ref{claim:path-cover}, we have $|L|  =\max\{1,s(G_1)\} \le \alpha(G_1) \le \alpha(G) \le \frac{n}{t+1}$. 
 	Since  $\kappa(G_2^*) \ge \kappa(G_2)$ and $\alpha(G_2) \le \alpha(G) \le \frac{n}{t+1}$,  it suffices to show that 
 	 $\kappa(G_2) \ge \frac{2n}{t+1}$. 
 	
As $|V(G_2)|=|S_2|+|V(D_2)| \ge \frac{7n}{12}$, we assume that $G_2$ is not a complete graph. 
Let $W\subseteq V(G_2)$ be  a minimum cutset of $G_2$.  Suppose to the contrary that  $|W|<\frac{2n}{t+1}$. 
 As $G_2 \subseteq G$ and so is $(2P_2 \cup P_1)$-free, 
$G_2-W$ has at most two nontrivial components, and if $G_2-W$ has two nontrivial components, then $c(G_2-W)=2$. 
Furthermore, since $|V(G_1)\cup W| <\frac{7n}{12}$ and $|V(D_2)\setminus W|>\frac{5n}{12}$ , and $c(G-(V(G_1)\cup W)) =c(G_2-W)$,  the toughness of $G$ implies that 
$G_2-W$ has a nontrivial component, and one of the  nontrivial components contains an edge of $D_2$. 
We assume that $Q_1$ is a component of $G_2-W$   that contains an edge of $D_2$. 
Then as  $G$ is $(2P_2 \cup P_1)$-free and $D_1$ is nontrivial, we know that $D_2$ is $(P_2\cup P_1)$-free. 
Thus, all other components of $G_2-W$ contains only vertices from $S_2$.  Let $x\in S_2$ 
be contained in a component  $Q_2$ of $G_2-W$  that is not $Q_1$, 
Since $d_G(x, D_2) \ge \frac{2n}{t+1}$ by the definition of $S_2$ and $V(Q_2)\cap V(D_2)=\emptyset$, it follows that  $|W| \ge |N_G(x)\cap V(D_2)| \ge \frac{2n}{t+1}$, a contradiction to the assumption that $|W|<\frac{2n}{t+1}$. 
 \end{proof}

 By Claim~\ref{KappaofG2star} and Lemma~\ref{haggkvist_cycles}, $G_2^*$ has a Hamilton cycle $C^*$ such that $L \subseteq E(C^*)$. Replacing each edge $xy$ of  $L$ in the cycle $C^*$  
 by the path of $\mathcal{Q}$ with endvertices $x$ and $y$ gives a Hamilton cycle of $G$.

{\bf \noindent Case 2:  Any  $uv \in E(G)$ satisfies $|N_G(u) \cup N_G(v)| > \frac{5n}{12}$.}

\smallskip

Let
  $$S = \{v \in V(G) : d_G(v) < \frac{5n}{24}\} \quad \text{and}  \quad S_1=\{x\in S: d_G(x) <\frac{n}{t+1}\}.$$
  As we assumed that $\delta(G) \le \frac{n}{t+1}-1$, we know that $S_1\ne \emptyset$. 
By the assumption of Case 2, $S$ is an independent set in $G$.
As $G$ is $t$-tough, we have $|S| \le  \frac{n}{t+1}$. By Lemma~\ref{lem:star-matching}, 
there exists a $K_{1,2}$-matching $M$ in $G$  with vertices of  $S_1$ precisely as its centers. 
 
Let $G_2=G-S$ and $L=\{xy:  \text{$x$ and $y$ are endvertices of a path of $M$}\}$.  Let  $G_2^*$ be obtained from $G_2$
by adding all the edges of $L$ that are not already in $G_2$.  
We will find in $G_2^*$ a Hamilton cycle containing all edges of $L$, expand it to a cycle containing all vertices
of $G_2$ and $S_1$, and finally ``inserting'' the vertices of $S\setminus S_1$ into the cycle to get a Hamilton cycle of $G$. 

\begin{CLA}\label{claim:connectivity-G^*}
We have $\kappa(G_2^*) \ge |L|+\alpha(G_2^*)$. 
\end{CLA}

\proof  
Since  $\kappa(G_2^*) \ge \kappa(G_2)$ and $\alpha(G_2^*) \le \alpha(G) \le \frac{n}{t+1}$,  it suffices to show that 
$\kappa(G_2) \ge  |L| +\frac{n}{t+1}$.  
As $|V(G_2)|  \ge \frac{11n}{12}$, we assume that $G_2$ is not a complete graph. 
Let $W\subseteq V(G_2)$ be a  minimum cutset of $G_2$. Suppose to the contrary that  $|W|< |L| +\frac{n}{t+1}$.  

Consider first that  $G_2-W$ has a trivial component consisting of a single vertex $y$.  Then $d_G(y) >\frac{4n}{12}$
if $y\sim x$ for some $x\in S_1$ by the assumption of Case 2. As $|S \cup W| <\frac{3n}{t+1}$ and $N_G(y) \subseteq S\cup W$, we obtain a contradiction. Thus we assume that $y\not\sim x$
for any $x\in S_1$. Then we have  $d_{G_2}(y) \ge \frac{5n}{24} -|S\setminus S_1| =  \frac{5n}{24}-|S|+|L|$. 
Thus $|W| \ge \frac{5n}{24}-|S|+|L|$. Since $|S| \le \frac{n}{t+1}$, we get a contradiction to the assumption that $|W|< |L| +\frac{n}{t+1}$.   

Thus $G_2-W$ has no trivial component.  Since $G_2$ is $(2P_2\cup P_1)$-free, it follows that $G_2-W$ has exactly two components that  are nontrivial.  Let $Q_1$ and $Q_2$ be these two components. 
As $|S \cup W| < \frac{3n}{t+1}$ and $|N_G(u) \cup N_G(v)| >\frac{5n}{12}$ for all $uv \in E(G)$, we have $|V(Q_i)| > \frac{5n}{t+1} - \frac{3n}{t+1}\ge \frac{2n}{t+1}$.  Since $G$ is $(2P_2\cup P_1)$, each $Q_i$ is $(P_2\cup P_1)$-free. Thus by 
Lemma~\ref{propertiesofP2cupP1freegraphs}(4), we have $\delta(Q_i) \ge |V(Q_i)|  -\alpha(Q_i)  \ge  |V(Q_i)|  -\frac{n}{t+1} >\frac{1}{2}|V(Q_i)|$. 
By Lemma~\ref{dirac}, each $Q_i$ is Hamiltonian.
Thus  each $Q_i$ has  a set $M_i$ of at least $\frac{n}{t+1}$ independent edges. Let $x\in S_1$. 
 Take  any edge $x_1y_1\in M_1$ and 
any edge $x_2y_2\in M_2$,  as $G$ is $(2P_2 \cup P_1)$-free,  $x$ must be adjacent in $G$  to one of the four vertices $x_1,y_1,x_2, y_2$. 
As we can form at least $\frac{n}{t+1}$ such pairs of edges  with one from  $M_1$ and  the other from $M_2$, it follows that 
$d_G(x) \ge \frac{n}{t+1}$. This gives a contradiction to the definition of $S_1$.  
\qed

By Claim~\ref{claim:connectivity-G^*} and Lemma~\ref{haggkvist_cycles}, $G_2^*$ contains a Hamilton cycle $C^*$ containing all edges of $L$.
Replacing each edge $xy$ of  $L$ in the cycle $C^*$  
by the path of $M$ with endvertices $x$ and $y$ gives a  cycle $C^{**}$ of $G$ that covers all vertices of $G_2$ and $S_1$.  If $S\setminus S_1=\emptyset$, then $C^{**}$ is already a Hamilton cycle of $G$. Thus we assume that $S\setminus S_1 \ne \emptyset$. 

For each vertex $x\in S\setminus S_1$, we have $d_G(x) =d_{G}(x, G_2) \ge \frac{n}{t+1}$. By Lemma~\ref{lem:vertex-insetting}, 
we can extend $C^{**}$ to a Hamilton cycle of $G$ by inserting vertices of $S\setminus S_1$ one by one.   

The proof of Theorem~\ref{main} is now complete. 
\end{proof}


\begin{thebibliography}{10}
	
	\bibitem{AK2011}
	J.~Akiyama and M.~Kano.
	\newblock {\em Factors and factorizations of graphs}, volume 2031 of {\em
		Lecture Notes in Mathematics}.
	\newblock Springer, Heidelberg, 2011.
	\newblock Proof techniques in factor theory.
	
	\bibitem{Bauer2006}
	D.~Bauer, H.~Broersma, and E.~Schmeichel.
	\newblock Toughness in graphs -- a survey.
	\newblock {\em Graphs and Combinatorics}, 22(1):1--35, 2006.
	
	\bibitem{MR1336668}
	D.~Bauer, H.~J. Broersma, J.~van~den Heuvel, and H.~J. Veldman.
	\newblock Long cycles in graphs with prescribed toughness and minimum degree.
	\newblock {\em Discrete Math.}, 141(1-3):1--10, 1995.
	
	\bibitem{Tough-counterE}
	D.~Bauer, H.~J. Broersma, and H.~J. Veldman.
	\newblock Not every 2-tough graph is {H}amiltonian.
	\newblock {\em Discrete Appl. Math.}, 99(1-3):317--321, 2000.
	
	\bibitem{2k2-tough}
	H.~Broersma, V.~Patel, and A.~Pyatkin.
	\newblock On toughness and {H}amiltonicity of {$2K_2$}-free graphs.
	\newblock {\em J. Graph Theory}, 75(3):244--255, 2014.
	
	\bibitem{chvatal-tough-c}
	V.~Chv{\'a}tal.
	\newblock Tough graphs and {H}amiltonian circuits.
	\newblock {\em Discrete Math.}, 5:215--228, 1973.
	
	\bibitem{D1952}
	G.~A. Dirac.
	\newblock Some theorems on abstract graphs.
	\newblock {\em Proceedings of the London Mathematical Society}, s3-2(1):69--81,
	1952.
	
	\bibitem{P3-union-2P1}
	Y.~Gao and S.~Shan.
	\newblock Hamiltonian cycles in 7-tough {$(P_3 \cup 2 P_1)$}-free graphs.
	\newblock {\em Discrete Math.}, 345(12):Paper No. 113069, 7, 2022.
	
	\bibitem{Cycles-through-edges}
	R.~H\"aggkvist and C.~Thomassen.
	\newblock Circuits through specified edges.
	\newblock {\em Discrete Math.}, 41(1):29--34, 1982.
	
	\bibitem{HG2021}
	A.~Hatfield and E.~Grimm.
	\newblock Hamiltonicity of 3-tough {$(K_2\cup 3K_1)$}-free graphs.
	\newblock {\em \arxiv{2106.07083}}.
	
	\bibitem{P4-free-1-tough}
	H.~A. Jung.
	\newblock On a class of posets and the corresponding comparability graphs.
	\newblock {\em J. Combinatorial Theory Ser. B}, 24(2):125--133, 1978.
	
	\bibitem{MR3557210}
	B.~Li, H.~J. Broersma, and S.~Zhang.
	\newblock Forbidden subgraphs for {H}amiltonicity of 1-tough graphs.
	\newblock {\em Discuss. Math. Graph Theory}, 36(4):915--929, 2016.
	
	\bibitem{OS2021}
	K.~Ota and M.~Sanka.
	\newblock Hamiltonian cycles in 2-tough {$2K_2$}-free graphs.
	\newblock {\em J. Graph Theory}, 101(4):769--781, 2022.
	
	\bibitem{MR4676626}
	K.~Ota and M.~Sanka.
	\newblock Some conditions for hamiltonian cycles in 1-tough {$(K_2\cup
		kK_1)$}-free graphs.
	\newblock {\em Discrete Math.}, 347(3):Paper No. 113841, 6, 2024.
	
	\bibitem{1706.09029}
	S.~Shan.
	\newblock Hamiltonian cycles in 3-tough {$2K_2$}-free graphs.
	\newblock {\em J. Graph Theory}, 94(3):349--363, 2020.
	
	\bibitem{S2021}
	S.~Shan.
	\newblock Hamiltonian cycles in tough {$(P_2\cup P_3)$}-free graphs.
	\newblock {\em Electronic J. Combinatorics}, 28(1):Paper No. P1.36, 2021.
	
	\bibitem{shan2025p4unionp1}
	S.~Shan.
	\newblock Hamiltonian cycles in tough $({P_4} \cup {P_1})$-free graphs.
	\newblock {\em \arxiv{2504.08936}}, 2025.
	
	\bibitem{MR4455928}
	L.~Shi and S.~Shan.
	\newblock A note on hamiltonian cycles in 4-tough {$(P_2\cup kP_1)$}-free
	graphs.
	\newblock {\em Discrete Math.}, 345(12):Paper No. 113081, 4, 2022.
	
	\bibitem{MR4658427}
	L.~Xu, C.~Li, and B.~Zhou.
	\newblock Hamiltonicity of 1-tough {$(P_2\cup kP_1)$}-free graphs.
	\newblock {\em Discrete Math.}, 347(2):Paper No. 113755, 7, 2024.
	
\end{thebibliography}

\end{document}